\documentclass[11pt]{amsart}

\usepackage{amsmath,amsthm,amsfonts,amssymb,latexsym}

\usepackage{tikz}
\usetikzlibrary{decorations}
\usetikzlibrary{decorations.pathreplacing}
\usetikzlibrary{decorations.markings}

\usepackage{amsrefs}

\newtheorem{thm}{Theorem}[section]
\newtheorem{thm*}{Theorem}
\newtheorem{lemma}[thm]{Lemma}
\newtheorem{lemma*}[thm*]{Lemma}
\newtheorem{prop}[thm]{Proposition}
\newtheorem{defn}[thm]{Definition}

\DeclareMathOperator{\ran}{ran}

\DeclareMathOperator{\Aut}{Aut}

\newcommand{\from}{\colon}

\newcommand{\cantor}{2^\omega}

\newcommand{\N}{\mathbb{N}}
\newcommand{\R}{\mathbb{R}}

\newcommand{\Q}{\mathbb{Q}}
\newcommand{\Z}{\mathbb{Z}}

\renewcommand{\subset}{\subseteq}
\renewcommand{\supset}{\supseteq}

\newcommand{\union}{\cup}
\newcommand{\bigunion}{\bigcup}
\newcommand{\inters}{\cap}

\newcommand{\biginters}{\bigcap}

\renewcommand{\and}{\,\,\&\,\,}

\renewcommand{\>}{\rangle}

\newcommand{\lecturedate}[1]{} 
\newcommand{\define}[1]{\textbf{#1}} 

\renewcommand{\cantor}{2^\N}

\begin{document}

\author{Andrew S. Marks}
\thanks{This work was done while visiting Benjamin Miller during Spring
2011. The author would like to think Benjamin Miller for his 
hospitality and helpful discussions, and also Clinton Conley, Alekos
Kechris, Fran\c{c}ois Le Ma\^{i}tre, Robin Tucker-Drob, and Carsten
Szardenings for helpful conversations.}

\title[Topological generators]{Topological
generators for full groups of hyperfinite pmp equivalence relations}

\address[Andrew S. Marks]{University of California, Los Angeles}
\email{marks@math.ucla.edu}

\begin{abstract} We give an elementary proof that there are two topological
generators for the full group of every aperiodic hyperfinite probability
measure preserving Borel equivalence relation. Our proof explicitly
constructs topological generators for the orbit equivalence
relation of the irrational rotation of the circle, and then appeals to
Dye's theorem and a Baire category argument to conclude the general case.
\end{abstract}

\maketitle

\section{Introduction}

Le Ma\^{i}tre~\cites{MR3274561, 1405.4745} has completely
elucidated the relationship between the cost of an aperiodic probability measure preserving (pmp) equivalence
relation and its number of topological generators. 
His work answered 
a question of Kechris~\cite{MR2583950}*{Section 4.(D)}, and built on
earlier results of Miller~\cite{MR2583950}*{Section 4.(D)}, Kittrell and
Tsankov~\cite{MR2599891}, and Matui~\cite{MR3103094}.
In this note we give an elementary proof of the following theorem which is
an essential ingredient in Le Ma\^{i}tre's proofs. It is originally due to
Matui~\cite[Theorem 3.2]{MR3103094} in the ergodic case, and Le Ma\^{i}tre
\cite[Theorem 4.1]{1405.4745} in general. 

\begin{thm}[\cite{MR3103094}]\label{general}
  Let $E$ be an aperiodic hyperfinite pmp
  equivalence relation on a standard probability space $(X,\mu)$.
  Then there exists $T \in \Aut(X,\mu)$ generating $E$ and an 
  involution $U \in [E]$ of arbitrarily small support such that $T$ and $U$
  are topological generators for the full group $[E]$.
\end{thm}

Our definitions and notations follow~\cite{MR2583950}. A Borel equivalence
relation $E$ on a standard probability space $(X,\mu)$ is said to be
\define{aperiodic}
if every $E$-class is infinite, and \define{measure preserving} if every
partial Borel injection $T \from X \to X$ such that $x \mathrel{E} T(x)$
almost everywhere is $\mu$-measure preserving. 
An equivalence relation $E$ on $(X,\mu)$ is
\define{hyperfinite} if it is the orbit equivalence relation of a Borel
action of $\Z$ (see
\cite{MR2095154}*{Theorem 6.6}).
$\Aut(X,\mu)$ is the space of measure preserving
automorphisms of $(X,\mu)$, identifying elements that agree almost
everywhere. If $T \in \Aut(X,\mu)$, we will let $E_T$ note the equivalence
relation on $X$ generated by $T$.

If $E$ is a measure preserving Borel
equivalence relation on $(X,\mu)$, then the \define{full group} $[E]$ of $E$ is the
set of $T \in \Aut(X,\mu)$ such that $x \mathrel{E} T(x)$ almost
everywhere. If every $E$-class is countable, then the full group $[E]$ is a Polish group when equipped with the
\define{uniform topology} given by the metric $d(S,T) = \mu(\{x : S(x) \neq
T(x)\}$. Elements $T_0, T_1, \ldots \in \Aut(X,\mu)$ are said to be
\define{topological generators} for $[E]$ if the closure of $\<T_0, T_1,
\ldots\>$ in the uniform topology is equal to $[E]$.

\section{The irrational rotation of the circle}

\begin{defn}[\cite{MR2599891}]
Suppose $(X,\mu)$ is a standard probability space and $T \in \Aut(X,\mu)$.
If $A \subset X$ is a
set such that $A$ and $T(A)$ are disjoint, then define 
the
involution $T_A: X \to X$ as
follows:
\[T_A(x) = \begin{cases} T(x) & \text{ if $x \in A$} \\
T^{-1}(x) & \text{ if $x \in T(A)$} \\
x & \text{ otherwise}
\end{cases}
\]
\end{defn}

We note the following:

\begin{prop}\label{prop1} Suppose $T \in \Aut(X,\mu)$. Then 
$T \circ T_{A}  \circ T^{-1} = T_{T(A)}$ and hence $T_{T^k(A)} \in
\<T,T_{n,A}\>$ for every $k \in \Z$. \qed
\end{prop}

If $T \in \Aut(X,\mu)$, then to 
verify some set topologically generates $[E_T]$, it will suffice to show that we can approximate
elements of the form $T_A$. 
\begin{prop}[{\cite[Proposition 4.4]{MR2599891}}]\label{KT}
Suppose $(X,\mu)$ is a standard probability space, and $T \in \Aut(X,\mu)$.
Then $\{T_A : \text{$A \subset X$ is Borel}\}$ is a set of topological
generators for $[E_T]$. 
\end{prop}

To derive this from \cite[Proposition
4.4]{MR2599891}, note that $E_T$ is generated by
$\{T_A : \text{$A \subset X$ is Borel}\}$ and that every element of
$[E_{T_A}]$ is of the form $T_{A'}$ for some $A' \subset A$.

\begin{thm}[\cite{MR3103094}]\label{ergodic}
  Let $E$ be an ergodic aperiodic hyperfinite pmp
  equivalence relation on a standard probability space $(X,\mu)$.
  Then there exists $T \in \Aut(X,\mu)$ generating $E$ and an 
  involution $U \in [E]$ of arbitrarily small support such that $T$ and $U$
  are topological generators for the full group $[E]$.
\end{thm}
\begin{proof}
By Dye's theorem~\cite{MR0131516} (see \cite{MR2095154}*{Section 7}), all
ergodic aperiodic pmp actions of $\Z$ are orbit
equivalent (i.e. conjugate by a measure preserving bijection). Hence, it is enough to prove the theorem for a single such
action. Let $\Z$ act on the circle $\R/\Z$ via an irrational
rotation $T(x) = x + \alpha$ (that is, $\alpha$ is irrational). This action
preserves Haar measure. 
We will identify elements of $\R/\Z$ with their coset representatives
in $[0,1)$ and use the usual
ordering on $[0,1)$ to compare them. Fix any irrational
$\beta < \alpha$.  
Our two topological generators for $[E_T]$ will be $T$ and $U = 
T_{[0,\beta)}$.

By Proposition~\ref{KT}, we only need to verify that $T_A \in
\overline{\<T,U\>}$ for every Borel set $A \subset \R/\Z$. Since every such
$A$ can be approximated arbitrarily well by a disjoint union of small
intervals, it suffices to show that for sufficiently small $\epsilon > 0$, $\overline{\<T,U\>}$ contains $T_{[x,x +
\epsilon)}$ for every $x$. Hence, by Proposition~\ref{prop1}, it
suffices to show that $T_{(0,\epsilon)} \in \overline{\<T,U\>}$ for
sufficiently small $\epsilon$, since orbits of the irrational rotation are
dense. 

So suppose $\epsilon > 0$ is sufficiently small. 
Then 
$T_{[\epsilon,\beta + \epsilon)} \in \overline{\<T,U\>}$ by 
Proposition~\ref{prop1}, since there are $k \in \Z$ such that $k\alpha$ is
arbitrarily close to $\epsilon$ in $\R/\Z$. 
Hence, $T_{[0,\beta)} \circ
T_{[\epsilon,\beta + \epsilon)} = T_{[0,\epsilon) \union
[\beta,\beta+\epsilon)} \in \overline{\<T,U\>}$. Next, observe 
\[T_{[(k-1)\beta, (k-1)\beta + \epsilon)
\union [k\beta,k\beta + \epsilon)} = T_{[(k-1)\beta, (k-1)\beta +
\epsilon)} \circ T_{[k \beta, k\beta + \epsilon)} \in
\overline{\<T,U\>}\] 
for every $k \geq 1$ by 
Proposition~\ref{prop1}. Now we have a telescoping series: 
\[T_{[0, \epsilon) \union [\beta, \beta + \epsilon) }
\circ \ldots \circ T_{[(k-1)\beta, (k-1)\beta + \epsilon) \union [k\beta,k\beta +
\epsilon)} = T_{[0,\epsilon)} \circ T_{[k\beta, k\beta+\epsilon)} \in
\overline{\<T,U\>}\]
Since $\beta$ is irrational, there must exist some $k$ such that $k
\beta$ is arbitrarily close to $\epsilon$. So $T_{[0,\epsilon)
\union [\epsilon, 2\epsilon)} = T_{[0,2\epsilon)} \in \overline{\<T,U\>}$. 
\end{proof}

\section{The non-ergodic case}

We now consider $E$ that are not ergodic. We begin with 
the case of $n$ irrational rotations of the circle.
We will use this special case to prove Theorem~\ref{general}. Below, we
will let $n = \{0,\ldots, n-1\}$. 

\begin{lemma}\label{ncircles}
  Consider the standard probability space $(n \times
  \R/\Z,\nu \times \lambda)$ where $\nu$ is the uniform measure on $n$ and
  $\lambda$ is Haar measure on $\R/\Z$. Let $\alpha_0, \ldots,
  \alpha_{n-1}$ be irrational and linearly independent over $\Q$, and 
  let $T \from n \times \R/\Z \to n \times \R/\Z$ be the map $T((i,x)) =
  (i,x + \alpha_i)$. Suppose $\beta \in (0,1)$ is an
  irrational number smaller than $\alpha_0, \ldots, \alpha_{n-1}$. Then $T$
  and $U = T_{n \times [0,\beta)}$ are topological generators for the full group $[E_T]$. 
\end{lemma}
\begin{proof}
As in the proof of Theorem~\ref{ergodic}, it is enough to show that 
$\overline{\<T,U\>}$ contains
involutions of the form $T_{\{i\} \times [0,\epsilon)}$ for every $i \in n$
and sufficiently small $\epsilon$.

Suppose $\epsilon > 0$ is sufficiently small. 
Consider the vectors $e_0 = (1,0,0,\ldots)$, $e_1 = (0,1,0,0\ldots), \ldots$
in $(\R/\Z)^n$.
Since the $\alpha_i$ are irrational and linearly independent over $\Q$,
there exists $k \in \Z$ with $(k\alpha_0, \ldots, k\alpha_{n-1})$ 
arbitrarily close to $\epsilon e_i$. Hence, there are $k \in \Z$ so that $T^k(n \times [0,\beta))$
is arbitrarily close to $(n \setminus i) \times [0,\beta) \union i \times
[\epsilon, \beta + \epsilon)$. Hence, $T_{(n \setminus i) \times [0,\beta) \union i \times
[\epsilon, \beta + \epsilon)} \in \overline{\<T,U\>}$ and so $T_{i
\times [0,\epsilon) \union i \times [\beta, \beta + \epsilon)} \in \overline{\<T,U\>}$ by
composing with $U$. Following
the reasoning in the proof of Theorem~\ref{ergodic}, we then see that $T_{i
\times [0, 2 \epsilon)} \in \<T,U\>$ as desired. 
\end{proof}

Recall the usual product topology on Cantor
space $\cantor$ is generated by the basic open sets $N_{s} = \{x \in
\cantor : x \supset s\}$ for $s \in 2^{< \N}$.

\begin{proof}[Proof of Theorem~\ref{general}]
  Suppose $f \from \cantor \to (1/2,1)$ is any continuous function such that
  $f(x)$ is irrational for every $x \in \cantor$. Define $T^f
  \from \cantor \times \R/\Z \to \cantor \times \R/\Z$ by the map $T^f(x,y)
  = (x,y + f(x))$.
  By the classification of probability measure preserving actions of
  $\Z$ up to orbit equivalence, if $E$ satisfies the hypothesis of
  the theorem, then there exists some Borel probability measure $\mu$ on
  $\cantor$ such that $E$ is orbit equivalent to the equivalence relation
  $E_{T^f}$ generated by $T^f$ on the space $(\cantor \times \R/\Z, \mu \times
  \lambda)$. (For example, for ergodic $E$, $\mu$ is a measure
  concentrating on a single point)

  Consider the space $C(\cantor,(1/2,1))$ of continuous functions from
  $\cantor$ to $(1/2,1)$ equipped with the compact-open topology. Now it is
  easy to see that for comeagerly many $f \in C(\cantor,(1/2,1))$, $f(x)$
  is irrational for every $x$. Fix an irrational $\beta \in (0,1/2)$. We
  will show that for comeagerly many $f \in C(\cantor,(1/2,1))$, if $\mu$
  is any Borel probability measure on $\cantor$, then $T^f$ and $U_\beta^f
  = T^f_{\cantor \times [0,\beta)}$ are topological generators for the
  equivalence relation $[E_{T^f}]$ (working on the space $(\cantor \times
  R/\Z,\mu \times \lambda)$). This proves the theorem by the above paragraph.

  By Proposition~\ref{KT} it suffices to show that for
  arbitrarily small $\epsilon > 0$ and every $s \in 2^{< \N}$, for every
  Borel probability measure $\mu$ measure on $\cantor$, $T_{N_s \times
  [0,\epsilon)} \in \overline{\<T^f,U^f\>}$ for comeagerly many $f \in
  C(\cantor,(1/2,1))$.

  So fix $s \in 2^{< \N}$, and a sufficiently small $\epsilon > 0$. Suppose $U \subset
  C(\cantor,(1/2,1))$ is an open set and $\delta > 0$. We will show
  that there is an open subset $U' \subset U$ so that for every $f \in U'$
  and Borel probability measure $\mu$ on $\cantor$,
  there is an $W \in \<T^f,U^f\>$ such that $d(W,T^f_{N_s \times
  [0,\epsilon)}) < \delta$ in the uniform topology on $[E_{T^f}]$. 
  This implies that the set of
  $f \in C(\cantor,(1/2,1))$ such that $\<T^f,U^f\>$
  contains an element of distance at most $\delta$ from $T_{N_s \times
  [0,\epsilon)}$ is comeager. But since this is true for every $\delta >
  0$ this implies that $T_{N_s \times [0,\epsilon)} \in \<T^f,U^f\>$
  for comeagerly many $f$.
  
  If $t \in 2^{< \N}$, $\alpha \in \R$, and $\xi > 0$, define the basic
  open set $U_{t,\alpha,\xi} = \{f \in C(\cantor,(1/2,1)) : \forall x \in
  N_t (|f(x) - \alpha| < \xi)\}$. It is easy to see that we can find some
  $m \in \N$, $p \from 2^m \to (1/2,1)$ and $\xi > 0$ such that $\biginters_{t
  \in 2^m} U_{t,p(t),\xi}$ is a nonempty subset of $U$. We may
  assume $m \geq |s|$ and $p$ is such that $\ran(p)$ is a set of irrational
  numbers that are linearly independent over $\Q$. 
  
  Our idea is that every $f \in \biginters_{t \in 2^m} U_{t,p(t),\xi}$ is well
  approximated by the system of $2^m$ irrational rotations generated by
  $T^p
  \from 2^m \times \R/\Z \to 2^m \times \R/\Z$ where $T^p((t,x)) = x + p(t)$.
  Let $A = \bigunion_{\{t \in 2^m : t \supset s\}} t \times [0,\epsilon)$
  and $U^p = T^p_{2^m
  \times [0,\beta)}$. By Lemma~\ref{ncircles}, we can find some finite word
  $W^p \in \<T^p,U^p\>$ such $d(W^p,T^p_{A}) < \delta/2^m$ in the uniform topology on
  $[E_{T^p}]$.   Choose $\xi' < \xi$ to be sufficiently small so that if $p'
  \from 2^m \to [0,1)$ satisfies $\forall t \in 2^m |p(t) - p'(t)| < \xi'$,
  then if $W^{p'}$ is the same word as $W^p$ but in $\<T^{p'},U^{p'}\>$,
  then $d(W^{p'},T^{p'}_{A}) < \delta/2^m$. 
  Let $U' = \inters_{t \in 2^m} U_{t,p(t),\xi'}$. For any $f \in U'$, the same word $W^f$ in $\<T^f,U^f\>$ will witness that $d(W^f,T^f_{N_s \times
  [0,\epsilon)}) < \delta$ in the uniform topology on $[E_{T^f}]$ with
  respect to any measure $\mu$ on $\cantor$ that concentrates on a single
  point. (Note that since we are considering any point measure on $\cantor$,
  the bound of $\delta/2^m$ on $d(W^{p'},T^{p'}_{A})$ has to be enlarged by
  a factor of $2^m$). Finally, this implies that $d(W^f,T^f_{N_s \times
  [0,\epsilon)}) < \delta$ with respect to 
  any measure $\mu$ on $\cantor$.
  \end{proof}

\bibliography{references}

\end{document}